\documentclass[amssymb,10pt]{article}
\usepackage[pagewise]{lineno}
\usepackage{graphicx}
\usepackage{indentfirst}
\usepackage{amsmath}
\usepackage{color}

\usepackage{amssymb}
\usepackage{amsthm}
\usepackage{titletoc}
\usepackage{amssymb}
\usepackage{amsxtra}
\usepackage{mathrsfs}
\usepackage{amsmath,amstext,amsfonts,verbatim, url}
\usepackage{epstopdf}
\usepackage{subfig}
\usepackage{psfrag}

\usepackage{tikz}

\usepackage{xy}

\usepackage[colorlinks=true]{hyperref}
\usepackage{amsfonts}
\usepackage{mathrsfs}
\usepackage{amsthm,amsmath,amssymb,mathrsfs,amsfonts,graphicx,graphics,latexsym,exscale,cmmib57,dsfont,bbm,amscd,ulem}
\usepackage{color,soul}
\newtheorem{thm}{Theorem}[section]
\newtheorem{lem}[thm]{Lemma}
\newtheorem{pro}[thm]{Proposition}

\def\cal#1{\fam2#1}

\def\T{{\mathbb T}}

\def\bb{\begin}
\def\be{\begin{equation}}
	\def\ee{\end{equation}}
\def\bea{\begin{eqnarray}}
	\def\eea{\end{eqnarray}}
\def\beaa{\begin{eqnarray*}}
	\def\eeaa{\end{eqnarray*}}

\def\ifl{\iffalse}

\def\bb{\begin}
           \def\ea{\end{array}}
          \def\ec{\end{center}}
     \def\ed{\end{description}}
\def\be{\bb{equation}}        \def\ee{\end{equation}}
\def\bea{\bb{eqnarray}}       \def\eea{\end{eqnarray}}
\def\beaa{\bb{eqnarray*}}     \def\eeaa{\end{eqnarray*}}
 \def\et{\end{thebibliography}}

       \def\F{{\cal F}}

\def\Int{{\rm Int}}

\begin{document}

\title{Technical Developments of DA on $\mathbb{T}^3$}

\date{}
\maketitle

{\center
Hangyue Zhang

\smallskip

Department of Mathematics

Nanjing University

Nanjing 210093, China

\smallskip

\smallskip


\smallskip

}

\begin{abstract}
We constructed a DA on $\mathbb{T}^3$, which complements the work of Gan, Li, Viana, and Yang  (\cite{GanLiVianaYang2021}) by providing an example of a $C^\infty$-diffeomorphism with partial volume expansion, where $\dim(E^{cs}) = 2$. In contrast to their work, in \cite{GanLiVianaYang2021}, they provided an example of a non-invertible embedding in the case $\dim(E^{cs}) = 2$. The inverse map of the DA we constructed has a mostly expanding center  (\cite{AnderssonVasquez2018}). Using a similar approach, we can also construct a (nontrivial) mixed center (\cite{MiCaoYang2017, MiCao2021}). 
\end{abstract}

\section{Introduction}

For modified mappings of hyperbolic linear automorphisms on $\mathbb{T}^n$, we refer to them as DA (derived from Anosov). The first DA on $\mathbb{T}^2$ was introduced by Smale \cite{Smale1967} in 1967, while the DA on $\mathbb{T}^3$ was first studied by Mañé \cite{Mane1978} in 1978. Katok and Hasselblatt \cite{KatokHasselblatt1995} clearly constructed Smale's DA on $\mathbb{T}^2$ in 1995. 
Thus, DAs have provided a rich source of ideas for constructing examples.
In the setting of mostly contracting centers, Bonatti and Viana \cite{BonattiViana2000} constructed a DA on $\mathbb{T}^3$ that controls the central Lyapunov exponents in 2000.  They controlled the central Lyapunov exponents by using the classical Stirling’s formula.
Using techniques similar to those of Bonatti and Viana, Andersson and Vásquez \cite{AnderssonVasquez2018} constructed a DA on $\mathbb{T}^3$ whose center direction is mostly expanding.
Viana and Yang \cite{VianaYang2017} studied the maximum entropy measure for the DA on $\mathbb{T}^3$ in 2017. 
Ures, Viana, F. Yang, and J. Yang studied the maximal $u$-entropy measure of DA on $\mathbb{T}^3$ in \cite{UresVianaYang2024}. 

The construction of DA plays an indispensable role in providing non-trivial examples for theoretical exploration.
In this paper, we show that a DA on $\mathbb{T}^3$ possesses the property of partial volume expansion. 
The inverse of the DA we construct admits a mostly expanding center. Although a DA on $\mathbb{T}^3$ with a mostly expanding center was already provided by Andersson and Vásquez, our approach remains genuinely novel. Furthermore, our method can be generalized to control the (modified) central Lyapunov exponents of arbitrary DAs on $\mathbb{T}^n$, even allowing for modifications performed on several pairwise disjoint small neighborhoods.  For example, using this technique, we constructed a mixed DA on $\mathbb{T}^3$, which is the primary subject of study in the papers \cite{MiCaoYang2017, MiCao2021}. 

Mañé’s DA diffeomorphisms \cite{Mane1978} are topologically mixing and mostly expanding \cite{AnderssonVasquez2018}, whereas Smale’s DA diffeomorphisms \cite{Smale1967,KatokHasselblatt1995} is a non-transitive  and are mostly contracting \cite{GanLiVianaYang2021}. The non-transitive mixed DA systems on $\mathbb{T}^3$ considered in this paper may provide inspiration for future investigations into more complex dynamical systems. 
By the way, it is worth mentioning that, so far, this is the first example of a mixed non-trivial center in three dimensions (For details, see Theorem~\ref{main3}). Therefore,  our method has broad applications and can provide nontrivial examples that belong to the settings of \cite{CaoMiFiniteness, David, MiCaoYang2017, CMY, HuaYangYang2020}.

\section{Basic Definitions and Results}

In this paper, a diffeomorphism \( f : M \to M \) on a smooth Riemannian manifold \( M \) is {\it partially hyperbolic} if there exists a continuous, \( Df \)-invariant splitting of the tangent bundle \( TM = E^{uu}\oplus_{\succ}E^{cs}  \), along with constants \( c > 0 \) and \( \sigma > 1 \), such that:
\begin{itemize}
    \item For all \( v^u \in E^{uu} \) and \( n \geq 1 \),
    \[
    \|Df^n v^u\| \geq c \sigma^n \|v^u\|.
    \]

    \item For all unit vectors \( v^u \in E^{uu} \) and \( v^{cs} \in E^{cs} \), and for \( n \geq 1 \),
    \[
    \frac{\|Df^n v^u\|}{\|Df^n v^{cs}\|} \geq c \sigma^n.
    \]
\end{itemize}
Here, \( E^{uu} \) is the {\it unstable bundle}, characterized by uniform expansion, while \( E^{cs} \), the {\it center-stable bundle}, is dominated by \( E^{uu} \). ($E_1 \oplus_{\succ} E_2$ means that $E_2$ is dominated by $E_1$.)  We call $\dim(E^{cs})$ the {\it $u$-codimension}.  A partially hyperbolic diffeomorphism \( f \) is \textit{partially volume-expanding} if
\[
|\det Df(x) |_H| > 1
\]
for any codimension-one subspace \( H \) of \( T_xM \) that contains \( E^{uu}_x \).

A probability measure $\mu$ is called a {\it Gibbs $u$-state} of $f$ if it is invariant and its conditional measures along the strong-unstable leaves of $f$ are absolutely continuous with respect to the Lebesgue measure on those leaves.
Given a $Df$-invariant subbundle $E$, we say that $E$ is {\it mostly expanding} (respectively, {\it mostly contracting)} if every Gibbs $u$-state has only positive (respectively, negative) Lyapunov exponents along $E$.

We obtain the following result.

\begin{thm}\label{main}
There exists a $C^\infty$ partially hyperbolic diffeomorphism $f$ on $\mathbb{T}^3$ with  a partially hyperbolic splitting
\[
T\mathbb{T}^3 = E^{uu}_f \oplus_{\succ} E^{c}_f \oplus_{\succ} F^{ss}_f,
\] such that $f$ is partially volume-expanding and  has a hyperbolic fixed point with unstable index 2. Moreover, there exists a $C^1$-neighborhood $\mathcal{U}_f$ of $f$ such that every diffeomorphism $g \in \mathcal{U}_f$ is also partially volume-expanding. 
\end{thm}

Furthermore, we establish the following result.

\begin{thm}\label{main2}
Let $f$ be as in Theorem~\ref{main}. Then the inverse map $f^{-1}$ has a mostly expanding center, that is, $E^c_{f^{-1}}$ ($=E^{c}_f$) is mostly expanding. Moreover, there exists a $C^1$-neighborhood $\mathcal{U}_{f^{-1}}$ of $f^{-1}$ such that every $C^{1+}$-diffeomorphism $h \in \mathcal{U}_{f^{-1}}$ has a mostly expanding center.
\end{thm}

To illustrate the broad applicability of our techniques, we present following Theorem~\ref{main3}.

\begin{thm}\label{main3}
There exists a $C^\infty$ partially hyperbolic diffeomorphism $G$ on $\mathbb{T}^3$ with a partially hyperbolic splitting
\[
T\mathbb{T}^3 = F^{uu}_G \oplus_{\succ} F^{cu}_G \oplus_{\succ} F^{cs}_G,
\]
such that $\dim(F^{uu}_G) =\dim(F^{cu}_G) = \dim(F^{cs}_G) = 1$, where $F^{cu}_G$ is mostly expanding (but not uniformly expanding) and $F^{cs}_G$ is mostly contracting (but not uniformly contracting). Moreover, all the above properties are $C^{1+}$-robust.
\end{thm}

The proof of Theorem~\ref{main} is embedded in subsection~\ref{volum} of Section ~\ref{aaa}. The proof of Theorem~\ref{main2} can be found in Section~\ref{bbb}, "Mostly Expanding Center." The proof of Theorem~\ref{main3} is provided in Section~\ref{ccc}, "Application: The Mixed Center Case."

\section{Construction of $f$}\label{aaa}

\subsection{Basic Setup}\label{setup}

Let
$$
D:\mathbb{T}^3 \to \mathbb{T}^3
$$
be the hyperbolic automorphism induced by
$$
\begin{pmatrix}  
2 & 1 & 1 \\  
1 & 1 & 1 \\  
1 & 1 & 0  
\end{pmatrix}.  
$$
We observe that $D$ admits a partially hyperbolic splitting
$$
T\mathbb{T}^3 = E^{uu} \oplus E^{s} \oplus E^{ss}.
$$
The foliations tangent to $E^{ss}$, $E^{s}$, and $E^{uu}$ are denoted by
$\mathcal{F}^{ss}(D)$, $\mathcal{F}^{s}(D)$, and $\mathcal{F}^{uu}(D)$, respectively. We equip $\mathbb{T}^3 = \mathbb{R}^3 / \mathbb{Z}^3$ with the Riemannian metric induced by the Euclidean metric on $\mathbb{R}^3$.  All determinants and curve lengths are computed with respect to this metric. 
There exists a fixed point $p$ and a sufficiently small $\delta > 0$ such that, for every $x \in \mathbb{T}^3$, the length of
\begin{equation}\label{slope}
\mathcal{F}^{ss}_{\tfrac{1}{4}}(x, D) \;\cap\; 
\big( \mathcal{F}^{uu}_{2\delta}(p, D) \times \mathcal{F}^{s}_{2\delta}(p, D) \times \mathcal{F}^{ss}_{2\delta}(p, D) \big)
\end{equation}
is at most $\tfrac{1}{200}$. (\textbf{Explanation:} We can view this in the lifted space with an isometry and point out that the quantity $\frac{1}{4}$ is not essential. What is crucial, however, is that the proportion   
$$
\frac{\operatorname{length}\left( \mathcal{F}^{ss}_{\sigma}(x, D) \cap \left( \mathcal{F}^{uu}_{2\delta}(p, D) \times \mathcal{F}^{s}_{2\delta}(p, D) \times \mathcal{F}^{ss}_{2\delta}(p, D) \right) \right)}{\operatorname{length}\left( \mathcal{F}^{ss}_{\sigma}(x, D) \right)}
$$ can be controlled by $\frac{1}{100}$ for some $\sigma,\delta>0$). We fix $\delta$ with this property.  {\color{black}Choose a $C^\infty$-smooth function $\psi \colon \mathbb{R} \to \mathbb{R}$ such that:
\begin{itemize}
\item $\psi(x) = \psi(-x)$ for all $x \in \mathbb{R}$ (i.e., $\psi$ is symmetric about $x = 0$),
\item $\psi(x)$ is strictly monotone on $\left( \frac{\delta}{2}, \delta \right)$,
\item $\psi(x) = 1$ for $x \in \left[ 0, \frac{\delta}{2} \right]$, and $\psi(x) = 0$ for $x \in [\delta, +\infty)$.
\end{itemize}
Once $\psi$ is fixed, it is non-zero only on a bounded closed set, and we have $x \, \psi'(x) \leq 0$.  Hence, there exists a constant $m > 0$ such that
\begin{equation}\label{budengshi1}
-m\le\bigl(x\, \psi'(x) + \psi(x)\bigr) \psi(y) \le 1 \quad \text{for all } x, y.
\end{equation}}

\subsection{Construction}

\begin{lem}\label{paowu}
There exist constants $\varepsilon_0 > 0$ and $M > 0$ such that
$$
\frac{c^2 + cu}{(1+\epsilon^2 + c^2)(1+\epsilon^2)} \ge M, \quad \text{for all } |u| \le \varepsilon_0, \ |\epsilon| \le \varepsilon_0, \text{ and } |c|\ge\frac{1}{100}.
$$
\end{lem}
The proof of this lemma is straightforward, and for completeness, we will leave the details to the appendix.

Let $A= D^{2n}$ be a hyperbolic linear automorphism on $\mathbb{T}^3$  for sufficiently large $n$, so that 
\begin{itemize}
    \item \( A \) has eigenvalues \( 0<\lambda_{ss} < \lambda_{s} < 1 <2<  \lambda_{uu}\)  such that
\begin{equation}\label{budengshi2}
\lambda_{ss} \cdot \lambda_s \cdot \lambda_{uu} = 1, \qquad \frac{M}{\lambda_s^2}>2, \qquad 
-\,m\Bigl(\frac{1}{2} - \frac{1}{\lambda_s}\Bigr) + \frac{1}{\lambda_s} \le \frac{1}{2 \lambda_{ss}},
\end{equation}
where $M$ is as in Lemma~\ref{paowu}. There exists a small constant $\kappa > 0$ such that
\begin{equation}\label{kappa}
(\frac{(1+\kappa^2)^{\frac{3}{2}}}{100}+\frac{1}{100})\log(\frac{1}{2})+(\frac{99}{100}-\frac{(1+\kappa^2)^{\frac{3}{2}}}{100})\log(\frac{1}{\lambda_s})>0.
\end{equation}
(\textbf{Explanation:} For any fixed $\lambda_s$ such that
$$
\left( \frac{1}{100} + \frac{1}{100} \right) \log\left(\frac{1}{2}\right) + \left( \frac{99}{100} - \frac{1}{100} \right) \log\left( \frac{1}{\lambda_s} \right) > 0.
$$
Since the function $\kappa \mapsto \left( \frac{(1 + \kappa^2)^{\frac{3}{2}}}{100} + \frac{1}{100} \right) \log\left( \frac{1}{2} \right) + \left( \frac{99}{100} - \frac{(1 + \kappa^2)^{\frac{3}{2}}}{100} \right) \log\left( \frac{1}{\lambda_s} \right)$ is continuous, relation~\ref{kappa} is always achievable.)
    \item The eigenvalues $\lambda_{ss}, \lambda_{s}, \lambda_{uu}$ correspond to mutually orthogonal eigenspaces $E^{ss}, E^{s}, E^{uu}$. The foliations that are tangent to these eigenspaces everywhere are denoted by \( \mathcal{F}^{ss}(A), \mathcal{F}^s(A), \mathcal{F}^{uu}(A) \), respectively.
    \item \( A \) has a fixed points, \( p \) with a open neighborhoods $U_p$ such that $\Lambda(p)$ is properly contained in $U_p$, where 
\[
\Lambda(p)=\Int(\F^{uu}_{2\delta}(p)\times \F^{s}_{2\delta}(p)\times \F^{ss}_{2\delta}(p)),\qquad  U_p=\Int(\F^{uu}_{4\delta}(p)\times \F^{s}_{4\delta}(p)\times \F^{ss}_{4\delta}(p)).
\]
\end{itemize}
 At this point, we have
$$
\mathcal{F}^{*}(A) = \mathcal{F}^{*}(D), \quad \text{for each } * \in \{uu, s, ss\}.
$$

First, for each $k\in\mathbb{N}$, we define $I_k$ as follows:
\begin{itemize}
\item For $(a,b,c) \in U_p$, set
$$
I_k^{-1}(a,b,c) = \Bigl(a, \lambda_s\cdot P(a,b,c), c\Bigr),
$$
where
{\color{black}$$
P(a,b,c) = \psi(kb) \cdot \psi\bigl(\sqrt{a^2 + c^2}\bigr) \cdot (\frac{1}{2} - \frac{1}{\lambda_s})b + \frac{1}{\lambda_s} b.
$$}
\item For $(a,b,c) \notin U_p$, set $I_k^{-1} = I$, the identity map.
\end{itemize}

Thus, 

\begin{lem}\label{zhongxin}
The map $I_k$ satisfies the following properties:
$$
\frac{1}{2} \le \frac{\partial P}{\partial b} \le \frac{1}{2 \lambda_{ss}}, \quad I_k(\Lambda(p)) = \Lambda(p), \quad 
\lim_{k \to +\infty} \frac{\partial P}{\partial a} = 0, \quad
\lim_{k \to +\infty} \frac{\partial P}{\partial c} = 0.
$$
Moreover, $I_k$ is a $C^\infty$ diffeomorphism.
\end{lem}
\begin{proof}
Let $r = \sqrt{a^2 + c^2}$. It then follows that
{\color{black}$$
\frac{\partial P}{\partial b} = (\frac{1}{2} - \frac{1}{\lambda_s})\, \psi(r) \bigl[ k b\, \psi'(k b) + \psi(k b) \bigr] + \frac{1}{\lambda_s}.
$$}
Since the map $X \mapsto \bigl(\frac{1}{2} - \frac{1}{\lambda_s}\bigr) X + \frac{1}{\lambda_s}$ is monotone decreasing, it follows from inequalities~\ref{budengshi1} and ~\ref{budengshi2} that
$$
\frac{1}{2} \le \frac{\partial P}{\partial b} \le \frac{1}{2 \lambda_{ss}}.
$$
For any $a \in \mathcal{F}^{uu}_{\delta}(p)$ and $c \in \mathcal{F}^{ss}_{\delta}(p)$, the map
{\color{black}$$
b \mapsto \lambda_s \cdot P(a,b,c) = \psi(kb) \cdot \psi\bigl(\sqrt{a^2 + c^2}\bigr) \cdot \Bigl(\frac{\lambda_s}{2} - 1\Bigr) b + b
$$}
is strictly increasing. Moreover,
$$
\lambda_s \cdot P(a, -\delta, c) = -\delta, \qquad \lambda_s \cdot P(a, \delta, c) = \delta.
$$
By connectedness, we have
$$
I_k^{-1}\bigl(\{a\} \times \mathcal{F}^{s}_{\delta}(p) \times \{c\}\bigr) = \{a\} \times \mathcal{F}^{s}_{\delta}(p) \times \{c\}.
$$
Then $I_k(\Lambda(p)) = \Lambda(p)$.
It is clear that $I_k^{-1} = I$ when $x \in U_p \setminus \Lambda(p)$ and $I_k=I$ when $x \in \T^3 \setminus \Lambda(p)$. 
{\color{black}It then follows from the smoothness of $\psi$ and \cite[Proposition~5.7]{Lee} that $I_k$ is a  $C^\infty$ diffeomorphism.} We can check that when $\frac{\partial P}{\partial a}$ and $\frac{\partial P}{\partial c}$ are both nonzero, we have
{\color{black}$$
\frac{\partial P}{\partial a} = \psi(kb) \cdot (\frac{1}{2} - \frac{1}{\lambda_s}) b \cdot \psi'(r) \cdot \frac{a}{r}, \qquad \frac{\partial P}{\partial c} = \psi(kb) \cdot (\frac{1}{2} - \frac{1}{\lambda_s}) b \cdot \psi'(r) \cdot \frac{c}{r}
$$
Since $\psi$ is nonzero only when $|kb| \le \delta$, it follows that $\psi$ is nonzero only when $|b| \le \frac{\delta}{k}$. Then, using the boundedness of $\psi(kb) \cdot (\frac{1}{2} - \frac{1}{\lambda_s})  \cdot \psi'(r)$, we obtain}
$$
\lim_{k \to +\infty} \frac{\partial P}{\partial a} = 0, \quad
\lim_{k \to +\infty} \frac{\partial P}{\partial c} = 0.
$$
\end{proof}

We now define $f_k$ by
{\color{black}$$
f_k := I_k\circ A.
$$}
Notice that 
$$
DI_k^{-1}= 
\begin{pmatrix}
1 & 0 & 0 \\[1mm]
\lambda_s \frac{\partial P}{\partial a} & \lambda_s \frac{\partial P}{\partial b} & \lambda_s \frac{\partial P}{\partial c} \\[1mm]
0 & 0 & 1
\end{pmatrix}:
E^{uu} \oplus E^s \oplus E^{ss} \to E^{uu} \oplus E^s \oplus E^{ss}.
 $$
Let $P_a = \frac{\partial P}{\partial a}, \, P_b = \frac{\partial P}{\partial b}, \, P_c = \frac{\partial P}{\partial c}$. It follows that:
{\color{black}$$
Df_k^{-1}|{\Lambda(p)} =  A^{-1}\circ DI_k^{-1}|{\Lambda(p)}  =
\begin{pmatrix}
\frac{1}{\lambda_{uu}} & 0 & 0 \\[2mm]
P_a & P_b & P_c \\[2mm]
0 & 0 & \frac{1}{\lambda_{ss}}
\end{pmatrix}:
E^{uu} \oplus E^s \oplus E^{ss} \to E^{uu} \oplus E^s \oplus E^{ss}.
$$
Thus, the tangent map of $f_k $ can be written as:
$$
Df_k|{f_k^{-1}(\Lambda(p))} =
\begin{pmatrix}
\lambda_{uu} & 0 & 0 \\[2mm]
-\frac{\lambda_{uu} P_a}{P_b} & \frac{1}{P_b} & -\frac{\lambda_{ss} P_c}{P_b} \\[2mm]
0 & 0 & \lambda_{ss}
\end{pmatrix}:
E^{uu} \oplus E^s \oplus E^{ss} \to E^{uu} \oplus E^s \oplus E^{ss}.
$$}

Notice that both $E^{uu}\oplus E^s$ and $E^{ss}\oplus E^s$ are invariant under $Df_k$.
Thus, we define  the unstable cone by
$$
\mathcal{C}_\alpha(E^{uu},E^s)
:=\bigl\{v=v^{uu}+v^{s}\;:\;v^{uu}\in E^{uu},\;v^{s}\in E^{s},\;\|v^{s}\|\le\alpha\|v^{uu}\|\bigr\}.
$$
Similarly, the stable cone is defined by
$$
\mathcal{C}_\alpha(E^{ss},E^s) 
:= \bigl\{ v = v^{ss} + v^{s} \;:\; v^{ss} \in E^{ss},\; v^{s} \in E^{s},\; \|v^{s}\| \le \alpha \|v^{ss}\| \bigr\}.
$$
It is well known that, in our setting, if there exists $\alpha>0$ such that
$$
Df_k\bigl(\mathcal{C}_\alpha(E^{uu},E^s)\bigr)\subset \mathcal{C}_\alpha(E^{uu},E^s)
\quad\text{and}\quad
Df^{-1}_k\bigl(\mathcal{C}_\alpha(E^{ss},E^s)\bigr)\subset \mathcal{C}_\alpha(E^{ss},E^s),
$$
then $f$ admits a partially hyperbolic splitting
$$
E^{uu}_k\oplus_{\succ} E^c_k \oplus_{\succ} E^{ss}_k,
$$
with $$
E_k^{uu}\subset \mathcal{C}_\alpha(E^{uu},E^s), \quad 
E_k^{ss}\subset \mathcal{C}_\alpha(E^{ss},E^s), 
\quad \text{and} \quad E_k^c = E^s,
$$ where $E_k^c \oplus E_k^{ss}$ corresponds to the center-stable bundle $E^{cs}$ in the definition of partial hyperbolicity.

\begin{lem}\label{jixian}
For any $\varepsilon>0$, there exists $K(\varepsilon)$ such that for every $k\ge K(\varepsilon)$,
 $$
Df_k\bigl(\mathcal{C}_\varepsilon(E^{uu},E^s)\bigr)\subset \mathcal{C}_\varepsilon(E^{uu},E^s)
\quad\text{and}\quad
Df_k^{-1}\bigl(\mathcal{C}_\varepsilon(E^{ss},E^s)\bigr)\subset \mathcal{C}_\varepsilon(E^{ss},E^s),
$$
Consequently, $f_k$ is partially hyperbolic for every $k \geq K(\varepsilon)$. 
\end{lem}
\begin{proof}
By Lemma~\ref{zhongxin} and the assumptions of our setting, the following inequality holds everywhere:
$$
\lambda_{uu} > \frac{1}{P_b} > \lambda_{ss},
$$
and 
$$
\lim_{k \to +\infty} \frac{\partial P}{\partial a} = 0, \quad
\lim_{k \to +\infty} \frac{\partial P}{\partial c} = 0.
$$
Thus, the lemma follows immediately.
\end{proof}

\begin{lem}
The point $p$ is a hyperbolic fixed point of $f_k$ with unstable index 2.
\end{lem}
\begin{proof}
Since
$$
Df_k(p) = 
\begin{pmatrix}
\lambda_{uu} & 0 & 0 \\[1mm]
0 & 2 & 0 \\[1mm] 
0 & 0 & \lambda_{ss} 
\end{pmatrix},
$$
it follows directly from this that the result stated in the lemma holds.
\end{proof}

In the following proof, we implicitly assume that $\varepsilon \leq \kappa$, where $\kappa$ is the constant given in inequality~\eqref{kappa} and $\varepsilon$ is as in Lemma~\ref{jixian}.

\subsection{Existence of $k$ such that $f_k$ is Partially Volume Expanding}\label{volum}

 Let
$$
v^{uu}=\frac{1}{\sqrt{1+\epsilon^2}}\left(1, \, \epsilon, \, 0\right)
$$
be a unit vector in $E^{uu}_k$. Keep in mind that
$$
E^{uu}_k \subset \mathcal{C}_\varepsilon(E^{uu}_A, E^s_A).
$$
Let $V$ be a two-dimensional linear subspace containing $E^{uu}_k$. Then $V$ can be written as
$$
V=\{\,x v^{uu} + y v : x,y \in \mathbb{R},\; v^{uu}\in E^{uu}_k,\; v \perp v^{uu}\,\}.
$$
The linear subspace orthogonal to $v^{uu}$   is
$$
\text{span}\big\{(-\epsilon,1,0),\; (0,0,1)\big\}.
$$
For simplicity, define
{\color{black}$$
Q_a := \lambda_{uu} P_a, \quad Q_c := \lambda_{ss} P_c.
$$}

We first prove that the following holds.

\begin{lem}\label{delta}
There exists a constant $K_1$ such that
$$
|\det(Df_k|_V)| > 1 \quad \text{for all } k \ge K_1,
$$
where
$$
V = \Bigl\{ x v^{uu} + y v : x, y \in \mathbb{R}, \; v^{uu} \in E^{uu}_k, \; v = \frac{1}{\sqrt{1 + \epsilon^2 + c^2}}(-\epsilon, \, 1, \, c), \; |c| \le \frac{1}{100} \Bigr\}.
$$
\end{lem}
\begin{proof}
Notice that the determinant of $Df_k$ restricted to the 2-dimensional subspace $V$ corresponds to the area expansion rate . Therefore,
$$
|\det\bigl(Df_k|_V\bigr)| = \|Df_k(v^{uu})\| \cdot \|Df_k(v)\| \cdot |\sin \theta|,
$$
where $\theta$ is the angle between the vectors $Df_k(v^{uu})$ and $Df_k(v)$. It follows that
$$
\begin{aligned}
\det\bigl(Df_k|_V\bigr)^2 &= \|Df_k(v^{uu})\|^2 \cdot \|Df_k(v)\|^2 \cdot \sin^2 \theta \\
     &= \|Df_k(v^{uu})\|^2 \cdot \|Df_k(v)\|^2 \cdot(1-\cos^2\theta) \\
     &= \|Df_k(v^{uu})\|^2 \cdot \|Df_k(v)\|^2-<Df_k(v^{uu}),Df_k(v)>^2
\end{aligned}
$$
Recall that
$$
Df_k =
\begin{pmatrix}
\lambda_{uu} & 0 & 0 \\[1mm]
-\dfrac{Q_a}{P_b} & \dfrac{1}{P_b} & -\dfrac{Q_c}{P_b} \\[1mm]
0 & 0 & \lambda_{ss} 
\end{pmatrix}:
E^{uu} \oplus E^s \oplus E^{ss} \to E^{uu} \oplus E^s \oplus E^{ss}.
$$
Direct calculation shows that:
$$
Df_k(v)=\frac{1}{\sqrt{1+\epsilon^2+c^2}}
\begin{pmatrix}
-\lambda_{uu} \epsilon \\[1mm]
\dfrac{1 + \epsilon Q_a - c Q_c}{P_b} \\[1mm]
\lambda_{ss} c
\end{pmatrix},\quad   Df_k(v^{uu})=\frac{1}{ \sqrt{1+\epsilon^2}}
\begin{pmatrix}
\lambda_{uu} \\[1em]
\dfrac{-Q_a + \epsilon}{P_b} \\[0.5em]
0
\end{pmatrix}.
$$
Then
$$
\begin{aligned}
\det\bigl(Df_k|_V\bigr)^2\cdot(1+\epsilon^2+c^2)\cdot (1+\epsilon^2) &=  \Big[(\lambda_{uu}\epsilon)^2 + (\frac{1 + \epsilon Q_a - c Q_c}{P_b})^2 + (\lambda_{ss} c)^2\Big] \cdot \Big[(\lambda_{uu})^2 + (\frac{-Q_a + \epsilon}{P_b})^2\Big] \\
&-\Big[-\lambda_{uu}^2 \epsilon + \frac{(1 + \epsilon Q_a - c Q_c)(-Q_a + \epsilon)}{P_b^2}\Big]^2\\
&=(\frac{1 + \epsilon Q_a - c Q_c}{P_b})^2\cdot(\lambda_{uu})^2 + (\lambda_{ss} c)^2\cdot (\lambda_{uu})^2\\
&+(\lambda_{uu}\epsilon)^2\cdot (\frac{-Q_a + \epsilon}{P_b})^2+(\lambda_{ss} c)^2\cdot(\frac{-Q_a + \epsilon}{P_b})^2\\
&+2\lambda_{uu}^2 \epsilon \cdot \frac{(1 + \epsilon Q_a - c Q_c)(-Q_a + \epsilon)}{P_b^2}\\
&\ge 4(\lambda_{ss} )^2\cdot (\lambda_{uu})^2\cdot (1 + \epsilon Q_a - c Q_c)^2 \\
&+ 8(\lambda_{uu})^2\cdot(\lambda_{ss})^2\cdot\epsilon\cdot(1 + \epsilon Q_a - c Q_c)(-Q_a + \epsilon)\\
&\quad \text{(Recall that } \frac{1}{P_b^2} \ge 4 (\lambda_{ss})^2 \text{ by Lemma~\ref{zhongxin}, and } \lambda_{ss} \cdot \lambda_s \cdot \lambda_{uu} = 1 \text{ by Relation~\ref{budengshi2})}\\
&\ge \frac{4\cdot (1 + \epsilon Q_a - c Q_c)^2}{\lambda_s^2}\\
&+\frac{1}{\lambda_s^2}\cdot  8\cdot\epsilon\cdot(1 + \epsilon Q_a - c Q_c)(-Q_a + \epsilon)
\end{aligned}
$$
By Lemma~\ref{zhongxin} and Lemma~\ref{jixian}, we obtain
$$
\lim_{k\to+\infty} Q_a = 0, 
\quad \lim_{k\to+\infty} Q_c = 0, 
\quad \text{and} \quad 
\lim_{k\to+\infty} \epsilon = 0.
$$
Additionally, given that $|c| \le \frac{1}{100}$, there exists a constant $K_1$ such that the conditions required by the lemma are satisfied.
\end{proof}

\begin{pro}\label{volume}
 There exists $K_3\ge K_1$ such that $f_k$ is partially volume expanding for all $k \ge K_3$.
\end{pro}
\begin{proof}

We will first explain a simple  case:
$$
v = (0,0,1).
$$
A direct computation shows that
$$
Df_k(v)=
\begin{pmatrix}
0 \\[1mm]
-\dfrac{Q_c}{P_b} \\[1mm]
\lambda_{ss}
\end{pmatrix}, \quad   Df_k(v^{uu})=\frac{1}{ \sqrt{1+\epsilon^2}}
\begin{pmatrix}
\lambda_{uu} \\[1em]
\dfrac{-Q_a + \epsilon}{P_b} \\[0.5em]
0
\end{pmatrix}.
$$
Analogous to the estimate in Lemma~\ref{delta}, we obtain that
$$
\begin{aligned}
\det\bigl(Df_k|_V\bigr)^2\cdot (1+\epsilon^2) &=\Bigl(\frac{Q_c^2}{P_b^2} + \lambda_{ss}^2\Bigr)
\Bigl(\lambda_{uu}^2 + \frac{(Q_a-\varepsilon)^2}{P_b^2}\Bigr)-\frac{Q_c^2}{P_b^2}\cdot \frac{(Q_a-\varepsilon)^2}{P_b^2} \\
&=\frac{Q_c^2}{P_b^2}\cdot \lambda_{uu}^2+\lambda_{ss}^2\cdot \lambda_{uu}^2+\lambda_{ss}^2\cdot\frac{(Q_a-\varepsilon)^2}{P_b^2}\\
&\ge \lambda_{ss}^2\cdot \lambda_{uu}^2\\
\end{aligned}
$$
Since $\lim_{k\to+\infty} \epsilon = 0$, there exists a constant $K_2 \ge K_1$ such that
$$
\det\bigl(Df_k|_V\bigr) > 1 \quad \text{for all } k \ge K_2,
$$
where
$$
V = \Bigl\{ x v^{uu} + y v : x, y \in \mathbb{R}, \; v^{uu} \in E^{uu}_k, \; v = (0,0,1) \Bigr\}.
$$
By combining Lemma~\ref{delta}, it is enough to show that there exists a constant $K_3 \ge K_2$ such that
$$
|\det(Df_k|_{V_1})| > 1 \quad \text{for all } k \ge K_3,
$$
where
$$
V_1 = \left\{ x v^{uu} + y v : x, y \in \mathbb{R}, \; v^{uu} \in E^{uu}_k, \; v = \frac{1}{\sqrt{1 + \epsilon^2 + c^2}} \left( -\epsilon, \, 1, \, c \right), \; |c| \ge \frac{1}{100} \right\}.
$$

Recalling the proof of Lemma~\ref{delta}, we have
$$
\begin{aligned}
\det\bigl(Df_k|_{V_1}\bigr)^2\cdot(1+\epsilon^2+c^2)\cdot (1+\epsilon^2) 
&=(\frac{1 + \epsilon Q_a - c Q_c}{P_b})^2\cdot(\lambda_{uu})^2 + (\lambda_{ss} c)^2\cdot (\lambda_{uu})^2\\
&+(\lambda_{uu}\epsilon)^2\cdot (\frac{-Q_a + \epsilon}{P_b})^2+(\lambda_{ss} c)^2\cdot(\frac{-Q_a + \epsilon}{P_b})^2\\
&+2\lambda_{uu}^2 \epsilon \cdot \frac{(1 + \epsilon Q_a - c Q_c)(-Q_a + \epsilon)}{P_b^2}\\
&\ge (\lambda_{ss} c)^2\cdot (\lambda_{uu})^2 \\
&+ 8(\lambda_{uu})^2\cdot(\lambda_{ss})^2\cdot[c\cdot\epsilon\cdot(-Q_c)(-Q_a + \epsilon)+\epsilon\cdot(1 + \epsilon Q_a)(-Q_a + \epsilon)]\\
&\quad \text{(Recall that } \frac{1}{P_b^2} \ge 4 (\lambda_{ss})^2 \text{ by Lemma~\ref{zhongxin}, and } \lambda_{ss} \cdot \lambda_s \cdot \lambda_{uu} = 1 \text{ by Relation~\ref{budengshi2})}\\
&\ge \frac{1}{\lambda_s^2}\cdot c^2+\frac{1}{\lambda_s^2}\cdot c\cdot 8\cdot\epsilon\cdot(-Q_c)(-Q_a + \epsilon)+\frac{1}{\lambda_s^2}\cdot 8\cdot\epsilon\cdot(1 + \epsilon Q_a)(-Q_a + \epsilon)
\end{aligned}
$$
Let $u=8\cdot\epsilon\cdot(-Q_c)(-Q_a + \epsilon)$ and $w=8\cdot\epsilon\cdot(1 + \epsilon Q_a)(-Q_a + \epsilon)$. 
By Lemma~\ref{zhongxin} and Lemma~\ref{jixian}, we obtain
$$
\lim_{k\to+\infty} w = 0, 
\quad \lim_{k\to+\infty} u = 0, 
\quad \text{and} \quad 
\lim_{k\to+\infty} \epsilon = 0.
$$
Then there exists $K_3\ge K_2$ such that, for every $k \ge K_3$, we have
$$
|\epsilon| \le \varepsilon_0\quad  |u| \le \varepsilon_0 \quad \text{and} \quad |w| \le \tfrac{\lambda_s^2}{2}.
$$
where $\varepsilon_0$ is given in Lemma~\ref{paowu}.
 It follows from Inequality~\ref{budengshi2}, by taking $k \ge K_3$, that
$$
\begin{aligned}
\det\bigl(Df_k|_{V_1}\bigr)^2 
&\;\ge\; \frac{1}{\lambda_s^2} \cdot 
\frac{c^2 + cu}{(1+\epsilon^2 + c^2)(1+\epsilon^2)} 
\;-\; \frac{1}{\lambda_s^2}\,|w| \\[2mm]
&\;\ge\; 2 \;-\; \frac{1}{\lambda_s^2}\,|w| \\
&\ge 2-\frac{1}{2}\\
&>1.
\end{aligned}
$$
\end{proof}

\subsection{Proof of Theorem~\ref{main}}

\begin{proof}[Proof of Theorem~\ref{main}]
By choosing $k \ge K_3$ and setting {\color{black}$f =I_k\circ A$}, where $K_3$ is as  in Proposition~\ref{volume}, we complete the construction. Since both partial hyperbolicity and partial volume expansion  are $C^1$-open properties, it follows that there exists a $C^1$-neighborhood $\mathcal{U}_f$ of $f$, and every $C^{1+}$-diffeomorphism in $\mathcal{U}_f$ is also partially volume expanding.
\end{proof}

\section{Mostly Expanding Center}\label{bbb}

\subsection{Basic Control Criterion}\label{criter}

We now set $g_k = f_k^{-1}$ and define
$$
\gamma_{uu} := \frac{1}{\lambda_{ss}}, \quad 
\gamma_{u} := \frac{1}{\lambda_{s}}, \quad 
\gamma_{ss} := \frac{1}{\lambda_{uu}}, 
\qquad 
F^{uu} := E^{ss}, \quad 
F^{ss} := E^{uu}, \quad 
F^{c} := E^{s}.
$$
Thus, on $F^{ss} \oplus F^{c} \oplus F^{uu}$, $Dg_k$ takes the form:
{\color{black}$$
Dg_k|_{\Lambda(p)} =  A^{-1}\circ DI_k^{-1}|_{\Lambda(p)} =
\begin{pmatrix}
\gamma_{ss} & 0 & 0 \\[2mm]
P_a & P_b & P_c \\[2mm]
0 & 0 & \gamma_{uu}
\end{pmatrix},
\qquad  
Dg_k|_{\mathbb{T}^3 \setminus \Lambda(p)} =
\begin{pmatrix}
\gamma_{ss} & 0 & 0 \\[2mm]
0 & \gamma_u & 0 \\[2mm]
0 & 0 & \gamma_{uu}
\end{pmatrix},
$$}
$$
\frac{1}{2} \;\le\; P_b \;\le\; \frac{\gamma_{uu}}{2}.
$$
For convenience, with respect to   
 $$ 
 F^{uu} \oplus F^{c} \oplus F^{ss}, 
 $$  
{\color{black} $Dg_k$ can be rewritten  as
$$
Dg_k|_{\Lambda(p)} =
\begin{pmatrix}
\gamma_{uu} & 0 & 0 \\[2mm]
P_c & P_b & P_a \\[2mm]
0 & 0 & \gamma_{ss}
\end{pmatrix},
\qquad  
Dg_k|_{\mathbb{T}^3 \setminus \Lambda(p)} =
\begin{pmatrix}
\gamma_{uu} & 0 & 0 \\[2mm]
0 & \gamma_u & 0 \\[2mm]
0 & 0 & \gamma_{ss}
\end{pmatrix},
$$}
For notational convenience, we assume $g_k = g$.   By our construction, it is clear that $g$ admits a partially hyperbolic splitting  
 $$ 
 T\mathbb{T}^3 = F^{uu}_g \oplus F^c_g \oplus F^{ss}_g, \quad \text{such that} \quad  F^{uu}_g \subset \mathcal{C}_\kappa(F^{uu}, F^c).
$$ 
It follows from the invariance of $F^c$ and the forward invariance of the cone that, for all $n \ge 1$,
\begin{equation}\label{contro}
\frac{1}{\sqrt{1+\kappa^2}}\gamma_{uu}^n \le \big|\det (Dg^n|_{F^{uu}_g})\big| \le \sqrt{1+\kappa^2}\,\gamma_{uu}^n.
\end{equation}
Alternatively, one can directly verify Relation~\ref{contro} by setting $v = v^{uu} + v^c$ such that
$$
\frac{1}{\sqrt{1+\kappa^2}} \le \frac{|v^{uu}|}{|v|} \le 1, \quad 
\begin{pmatrix}
{\color{black}\gamma_{uu}} & 0 \\
* & *
\end{pmatrix}
\begin{pmatrix}
v^{uu} \\
v^c
\end{pmatrix}
=
\begin{pmatrix}
{\color{black}\gamma_{uu}} v^{uu} \\
*
\end{pmatrix}
\in \mathcal{C}_{\kappa}(F^{uu}, F^c).
$$

{\color{black}  Since the Riemannian metric on $\mathbb{T}^3$ is induced by the Euclidean metric on $\mathbb{R}^3$, the lengths and volumes of unstable disks can be considered directly in $\mathbb{R}^3$.  The Riemannian metric on $\T^3$ also induces a (non-normalized) Riemannian volume on each leaf of the strong-unstable foliation, denoted by $m_L$ for any disk $L$ contained in a strong-unstable leaf of $g$.  For any measurable set $U$, $m_L(U)=m_L(U\cap L)$.}

 It is clear that
$$
g^n_*(m_L) = \big|\det(Dg|_{E^{uu}_y}^{-n})\big| \, m_{g^n(L)},
$$
which means that for any measurable set $B \subset g^n(L)$,
\begin{equation}\label{ml}
g^n_*(m_L)(B) = \int_B \big|\det(Dg|_{E^{uu}_y}^{-n})\big| \, dm_{{\color{black}g^n(L)}}.
\end{equation}
It follows from relations~\ref{contro} and~\ref{ml}  that
\begin{equation}\label{contro3}
\frac{ m_{g^n(L)}(B)}{\sqrt{1+\kappa^2}\,\gamma_{uu}^n}\le g^n_*(m_L)(B)\le \frac{m_{g^n(L)}(B)\sqrt{1+\kappa^2}}{\gamma_{uu}^n}.
\end{equation}

By \cite{HPS}, it is clear that, for any strong-unstable disk $U$ of $g$ and any $x \in U$,
$$
U \subset \mathcal{F}^{uu}(x, A^{-1}) \times \mathcal{F}^{u}(x, A^{-1}),
$$
where $\mathcal{F}^{uu}(x, A^{-1}) = \mathcal{F}^{ss}(x, A)=\mathcal{F}^{ss}(x, D)$ and $\mathcal{F}^{u}(x, A^{-1}) = \mathcal{F}^{s}(x, A)=\mathcal{F}^{s}(x, D)$. Define
$$
\pi_{uu}^x : \mathcal{F}^{uu}(x, A^{-1}) \times \mathcal{F}^{u}(x, A^{-1}) \to \mathcal{F}^{uu}(x, A^{-1}) \quad \text{by} \quad \pi_{uu}^x(a,b) = a.
$$
At this stage, we decompose $g^n(L)$ into finitely many mutually disjoint segments
$$
g^n(L) = L_1 \cup L_2 \cup \cdots \cup L_{k(n)} \cup L(n)
$$
such that for each $i = 1, 2, \dots, k(n)$,  there exists some point $x \in L_i$:
$$
\pi_{uu}^x(L_i) = \F^{uu}_{\frac{1}{4}}(x, A^{-1}),
$$
while the remaining segment $L(n)$ satisfies, for every $x \in L(n)$,
$$
0 \le |\pi_{uu}^x(L(n))| < \frac{1}{2}.
$$
It follows from relation~$\ref{slope}$ that, for each $i$,
$$
\frac{m_{L_i}(L_i \cap \Lambda(p))}{m_{L_i}(L_i)} \le \frac{\sqrt{1+\kappa^2}}{100}.
$$
We provide the following cross-sectional diagram for better understanding.
\begin{equation}\label{varep}
	\includegraphics[width=0.8\textwidth]{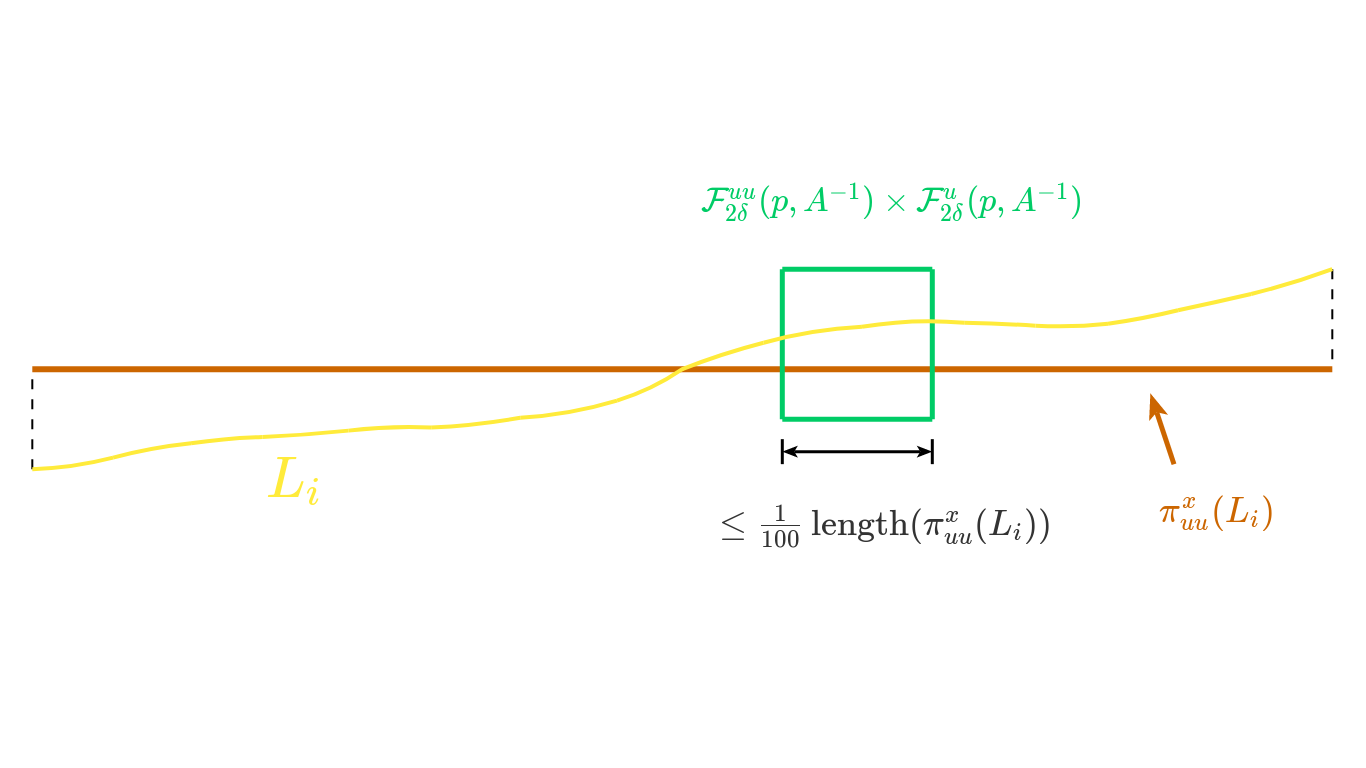}
\end{equation}
In particular,
$$ 
m_{L(n)}(L(n))\le \frac{\sqrt{1+\kappa^2}}{2}. 
$$

Since $$m_{g^n(L)}=m_{L(n)}+\sum_{1\le i\le k(n)}m_{L_i},$$  it follows that
\begin{equation}\label{contro2}
\frac{m_{g^n(L)}(\Lambda(p))}{m_{g^n(L)}(g^n(L))}\le\frac{\sum_{1\le i\le k(n)}m_{L_i}(\Lambda(p))+ \frac{\sqrt{1+\kappa^2}}{2}}{\sum_{1\le i\le k(n)}m_{L_i}(L_i)}\le\frac{\sqrt{1+\kappa^2}}{100}+\frac{\frac{\sqrt{1+\kappa^2}}{2}}{m_{g^n(L)}(g^n(L))-\frac{\sqrt{1+\kappa^2}}{2}}
\end{equation}
Since  the Lebesgue measure of a $C^1$ curve coincides with its length, it follows from relation~$\ref{contro}$ that 
$$
\frac{1}{\sqrt{1+\kappa^2}}\gamma_{uu}^n\operatorname{length}(L)  \le m_{g^n(L)}(g^n(L)) = \operatorname{length}(g^n(L))\le  \sqrt{1+\kappa^2}\,\gamma_{uu}^n\operatorname{length}(L).
$$
It follows from relation~\ref{contro2} that
$$
{\color{black}m_{g^n(L)}(\Lambda(p))}\le \frac{1+\kappa^2}{100}\gamma_{uu}^n\operatorname{length}(L)+{\color{black}\frac{\frac{\sqrt{1+\kappa^2}}{2}}{1-\frac{1+\kappa^2}{2\gamma_{uu}^n\operatorname{length}(L)}}.}
$$
Combining this with relation~$\ref{contro3}$, we then obtain
$$
g^n_*(m_L)(\Lambda(p))\le \frac{m_{g^n(L)}(\Lambda(p))\sqrt{1+\kappa^2}}{\gamma_{uu}^n}\le \frac{(1+\kappa^2)^{\frac{3}{2}}}{100}\operatorname{length}(L)+\frac{\frac{1+\kappa^2}{2}}{\gamma_{uu}^n-{\color{black}\frac{1+\kappa^2}{2\operatorname{length}(L)}}}
$$
Then we have 
\begin{equation}\label{contro0}
\frac{1}{\operatorname{length}(L)}g^n_*(m_L)(\Lambda(p))\le \frac{(1+\kappa^2)^{\frac{3}{2}}}{100}+\frac{\frac{1+\kappa^2}{2}}{\gamma_{uu}^n\operatorname{length}(L)-\frac{{\color{black}1+\kappa^2}}{2}}
\end{equation}

(Please keep in mind that relation~\ref{contro0} is derived only from relation~\ref{contro} and relation~\ref{slope}.) 

\subsection{Proof of Theorem~\ref{main2}}

\begin{proof}[Proof of Theorem~\ref{main2}]
Since the ergodic components of Gibbs $u$-states are  Gibbs $u$-states (from \cite[Section~11.2]{CLM}), it suffices to prove that every ergodic Gibbs $u$-state has only positive Lyapunov exponents along $F^c$. By the Birkhoff ergodic theorem, for an ergodic Gibbs $u$-state $\mu$, having only positive Lyapunov exponents along $F^c$ is equivalent to
$$
\int \log \det \big(Dg|_{F^c}\big)\, d\mu > 0.
$$
It remains to show that
$$
\int \log \det \big(Dg|_{F^c}\big)\, d\mu > 0.
$$
Combining the definition of a Gibbs $u$-state with ergodicity, we obtain that there exists a disk $L$ contained in a strong-unstable leaf of $g$ such that
$$
\lim_{\ell \to +\infty} \frac{1}{\ell} \sum_{n=0}^{\ell-1} \frac{g^n_* m_L}{m_L(L)} = \mu.
$$
By relation~$\ref{contro0}$ and the identity $\operatorname{length}(L) = m_L(L)$, there exists a sufficiently large $N$ such that for every $n \geq N$, 
$$
\frac{g^n_* m_L(\Lambda(p))}{m_L(L)} \leq \frac{(1+\kappa^2)^{3/2}}{100} + \frac{1}{100}.
$$
It follows from the openness of $\Lambda(p)$ that
$$
\mu(\Lambda(p)) \;\leq\; \liminf_{\ell \to +\infty} \frac{1}{\ell} \sum_{n=0}^{\ell-1} \frac{g^n_* m_L}{m_L(L)}(\Lambda(p)) \;\leq\; \frac{(1+\kappa^2)^{3/2}}{100} + \frac{1}{100}.
$$
Therefore, it follows from relation~$\ref{kappa}$ that
$$
\int_{\Lambda(p)\cup (\T^3\setminus\Lambda(p))}\log(\det Dg|_{F^c})d\mu\ge (\frac{(1+\kappa^2)^{\frac{3}{2}}}{100}+\frac{1}{100})\log(\frac{1}{2})+(\frac{99}{100}-\frac{(1+\kappa^2)^{\frac{3}{2}}}{100})\log(\gamma_u)>0.
$$
\end{proof}

We remark that some of the ideas presented here are inspired by the construction of the four-dimensional mixed center in \cite{Zhang}. However, we generalize the methods in \cite{Zhang}, meaning that the weight (with respect to the ergodic Gibbs $u$-state) of the modified regions can be set arbitrarily small.

\section{Application: The Mixed Center Case}\label{ccc}
Let
$$
C:\mathbb{T}^3 \to \mathbb{T}^3
$$
be the hyperbolic automorphism induced by
$$
\begin{pmatrix}
1 & -1 & 0\\
-1 & 1 & 1\\
0 & 1 & -1
\end{pmatrix},
$$
and  $C$ admits a partially hyperbolic splitting
$$
T\mathbb{T}^3 = E^{uu} \oplus E^{u} \oplus E^{ss}.
$$
(It can be directly verified that $C$ is the inverse map of $D$.)
The foliations tangent to $E^{ss}$, $E^{s}$, and $E^{uu}$ are denoted by
$\mathcal{F}^{ss}(C)$, $\mathcal{F}^{s}(C)$, and $\mathcal{F}^{uu}(C)$, respectively. Analogously to Subsection~\ref{setup}, we equip $\mathbb{T}^3 = \mathbb{R}^3 / \mathbb{Z}^3$ with the Riemannian metric induced by the Euclidean metric on $\mathbb{R}^3$.  All determinants and curve lengths are computed with respect to this metric. 
There exist two fixed points $q_1$ and $q_2$, and a sufficiently small $\delta > 0$, such that for every $x \in \mathbb{T}^3$ the length of
\begin{equation}\label{slope2}
\mathcal{F}^{uu}_{\tfrac{1}{4}}(x, C) \;\cap\; 
\big( \mathcal{F}^{uu}_{2\delta}(q_i, C) \times \mathcal{F}^{u}_{2\delta}(q_i, C) \times \mathcal{F}^{ss}_{2\delta}(q_i, C) \big) \quad \text{for each } i=1,2
\end{equation}
is at most $\tfrac{1}{200}$ and 
$$
\big( \mathcal{F}^{uu}_{5\delta}(q_1, C) \times \mathcal{F}^{u}_{5\delta}(q_1, C) \times \mathcal{F}^{ss}_{5\delta}(q_1, C) \big)\cap \big( \mathcal{F}^{uu}_{5\delta}(q_2, C) \times \mathcal{F}^{u}_{5\delta}(q_2, C) \times \mathcal{F}^{ss}_{5\delta}(q_2, C) \big)=\emptyset.
$$
We fix $\delta$ with above property. {\color{black} Choose a $C^\infty$-smooth function $\phi \colon \mathbb{R} \to \mathbb{R}$ such that:
\begin{itemize}
\item $\phi(x) = \phi(-x)$ for all $x \in \mathbb{R}$ (i.e., $\phi$ is symmetric about $x = 0$),
\item $\phi(x)$ is strictly monotone on $\left( \frac{\delta}{2}, \delta \right)$,
\item $\phi(x) = 1$ for $x \in \left[ 0, \frac{\delta}{2} \right]$, and $\phi(x) = 0$ for $x \in [\delta, +\infty)$.
\end{itemize}
Once $\phi$ is fixed, it is non-zero only on a bounded closed set, and we have $x \, \phi'(x) \leq 0$.  Hence, there exists a constant $m > 0$ such that
\begin{equation}\label{budeng}
-m\le\bigl(x\, \phi'(x) + \phi(x)\bigr) \phi(y) \le 1 \quad \text{for all } x, y.
\end{equation}
}

Now let $B= C^{2n}$ be a hyperbolic linear automorphism on $\mathbb{T}^3$  for sufficiently large $n$, so that 
\begin{itemize}
    \item \( B \) has eigenvalues \( 0<\lambda_{ss}<\frac{1}{2} < 2<\lambda_{u} <   \lambda_{uu}\)  such that
\begin{equation}\label{budeng2}
\lambda_{ss} \cdot \lambda_s \cdot \lambda_{uu} = 1,  \qquad 
-\,m\Bigl(\frac{1}{2} -\lambda_u\Bigr) + \lambda_u \le \frac{\lambda_{uu}}{2 }.
\end{equation}
 There exists a small constant $\kappa_2 > 0$ such that
\begin{equation}\label{kappa2}
(\frac{(1+\kappa^2_2)^{\frac{3}{2}}}{100}+\frac{1}{100})\log(\frac{1}{2}-\kappa_2)+(\frac{99}{100}-\frac{(1+\kappa_2^2)^{\frac{3}{2}}}{100})\log(\frac{1}{\sqrt{1+\kappa_2^2}} \lambda_{u} )>0,
\end{equation}
and 
{\color{black}\begin{equation}\label{kappa33}
(\frac{(1+\kappa^2_2)^{\frac{3}{2}}}{100}+\frac{1}{100})\log(2+\kappa_2)+(\frac{99}{100}-\frac{(1+\kappa_2^2)^{\frac{3}{2}}}{100})\log(\sqrt{1+\kappa_2^2} {\color{black}\lambda_{ss}})<0.
\end{equation}
(Analogously to relation~\ref{kappa}, \ref{kappa2} and \ref{kappa33} can always be achieved.)}
    \item The eigenvalues $\lambda_{ss}, \lambda_{u}, \lambda_{uu}$ correspond to mutually orthogonal eigenspaces $E^{ss}, E^{u}, E^{uu}$. The foliations that are tangent to these eigenspaces everywhere are denoted by \( \mathcal{F}^{ss}(B), \mathcal{F}^u(B), \mathcal{F}^{uu}(B) \), respectively.
    \item $B$ has two fixed points $q_i$, each admitting an open neighborhood $U_{q_i}$ such that $\Lambda(q_i)$ is properly contained in $U_{q_i}$, where
$$
\Lambda(q_i) = \operatorname{Int}\big( \mathcal{F}^{uu}_{2\delta}(q_i) \times \mathcal{F}^{u}_{2\delta}(q_i) \times \mathcal{F}^{ss}_{2\delta}(q_i) \big),  
\qquad 
U_{q_i} = \operatorname{Int}\big( \mathcal{F}^{uu}_{4\delta}(q_i) \times \mathcal{F}^{u}_{4\delta}(q_i) \times \mathcal{F}^{ss}_{4\delta}(q_i) \big). 
$$
\end{itemize}
 At this point, we have
$$
\mathcal{F}^{*}(B) = \mathcal{F}^{*}(C), \quad \text{for each } * \in \{uu, u, ss\}.
$$

In the following definition, we regard each fixed point $q_i$ as the origin $(0,0,0)$. Assume, without loss of generality, that $U_{q_1}$ and $U_{q_2}$ are disjoint.
Now we define the map $J_k$ as follows:
{\color{black}\begin{itemize}
\item For $(a,b,c) \in U_{q_1}$, set
$     J_k(a,b,c) = \Bigl(a, \,  \frac{1}{\lambda_u}\cdot Q_1(a,b,c), \, c\Bigr),
    $
where
$$  Q_1(a,b,c) = \phi(kb)\cdot \phi\bigl(\sqrt{a^2 + c^2}\bigr)\cdot \Bigl(\tfrac{1}{2} - \lambda_u\Bigr)b + \lambda_u b.
$$
\item For $(a,b,c) \in U_{q_2}$, set
$     J_k^{-1}(a,b,c) = \Bigl(a, \,  b, \, \lambda_{ss}\cdot Q_2(a,b,c)\Bigr),
    $
where
$$  Q_2(a,b,c) = \phi(kc) \cdot \phi\bigl(\sqrt{a^2 + b^2}\bigr) \cdot (\frac{1}{2} - \frac{1}{\lambda_{ss}})c + \frac{1}{\lambda_{ss}} c.
    $$

\item For $(a,b,c) \notin U_{q_1}\cup U_{q_2}$, set $J_k = I$, the identity map.
\end{itemize}}

Notice that $U_{q_1}$ and $U_{q_2}$ are disjoint.
The map $J_k$ coincides with the identity map on $\T^3\setminus (U_{q_1}\cup U_{q_2})$.  Similar to Lemma~\ref{zhongxin}, we have the following lemma.

\begin{lem}\label{mix}
The map $J_k$ satisfies the following properties:
$$
\frac{1}{2} \le \frac{\partial Q_1}{\partial b} \le \frac{\lambda_{uu}}{2}, \quad 
J_k(\Lambda(q_1)) = \Lambda(q_1), \quad 
\lim_{k \to +\infty} \frac{\partial Q_1}{\partial a} = 0, \quad
\lim_{k \to +\infty} \frac{\partial Q_1}{\partial c} = 0.
$$
$$
\frac{1}{2} \le \frac{\partial Q_2}{\partial c}, \quad J_k(\Lambda(q_2)) = \Lambda(q_2), \quad 
\lim_{k \to +\infty} \frac{\partial Q_2}{\partial a} = 0, \quad
\lim_{k \to +\infty} \frac{\partial Q_2}{\partial b} = 0.
$$
Moreover, $J_k$ is a $C^\infty$ diffeomorphism.
\end{lem}

With the above preparation, we now define the map
$$
G_k := B \circ J_k.
$$

By a straightforward computation, we obtain:
{\color{black}$$
DG_k|_{\Lambda(q_1)}=
\begin{pmatrix}
\lambda_{uu} & 0 & 0\\[3pt]
 \frac{\partial Q_1}{\partial a} &  \frac{\partial Q_1}{\partial b}  & \frac{\partial Q_1}{\partial c}\\[3pt]
0 & 0 & \lambda_{ss}
\end{pmatrix},\quad   DG_k|_{\Lambda(q_2)}=\begin{pmatrix}
\lambda_{uu} & 0 & 0\\[6pt]
0 & \lambda_{u} & 0\\[8pt]
* &
* &
 (\frac{\partial Q_2}{\partial c})^{-1}
\end{pmatrix}
$$
and
$$
DG_k^{-1}|_{G_k(\Lambda(q_1))}=
\begin{pmatrix}
\frac{1}{\lambda_{uu}} & 0 & 0\\[5pt]
* &  (\frac{\partial Q_1}{\partial b})^{-1}  & *\\[5pt]
0 & 0 & \frac{1}{\lambda_{ss}}
\end{pmatrix}
\quad 
DG_k^{-1}|_{G_k(\Lambda(q_2))}=
\begin{pmatrix}
\frac{1}{\lambda_{uu}} & 0 & 0\\[4pt]
0 & \frac{1}{\lambda_{u}} & 0\\[6pt]
* & * & \frac{\partial Q_2}{\partial c}
\end{pmatrix}.
$$ (The inverse map is considered to ensure that
$$
F^{cs}_k \subset \mathcal{C}_\varepsilon(E^s, E^u),
$$
as shown in following Lemma~\ref{jixian2}.)}
Since $J_k$ coincides with the identity map on 
$\mathbb{T}^3 \setminus \big(\Lambda(q_1) \cup \Lambda(q_2)\big)$,
it follows that $DG_k$ agrees with the original B on $\mathbb{T}^3 \setminus \big(\Lambda(q_1) \cup \Lambda(q_2)\big)$. {\color{black} In fact, one can verify that the entries marked with an asterisk $(*)$ are respectively given by
$$
\frac{\partial Q_2}{\partial a}\cdot c_1, \quad 
\frac{\partial Q_2}{\partial b}\cdot c_2, \quad 
\frac{\partial Q_1}{\partial a}\cdot b_1, \quad 
\frac{\partial Q_1}{\partial c}\cdot b_2, \quad
\frac{\partial Q_2}{\partial a}\cdot d_1, \quad 
\frac{\partial Q_2}{\partial b}\cdot d_2,
$$
where $c_1, c_2, b_1, b_2, d_1, d_2$ are functions defined on the corresponding domains, satisfying
$$
|c_1|+|c_2|+|b_1|+|b_2|+|d_1| + |d_2| \le \xi
$$
for some constant $\xi > 0$.}
By Lemma~\ref{mix} and the uniqueness of dominated splitting, similar to Lemma~\ref{jixian}, we obtain the following result.

\begin{lem}\label{jixian2}
For any $\varepsilon>0$, there exists $K(\varepsilon)$ such that for every $k\ge K(\varepsilon)$,
 $$
DG_k\bigl(\mathcal{C}_\varepsilon(E^{uu},E^u\oplus E^s)\bigr)\subset \mathcal{C}_\varepsilon(E^{uu},E^u\oplus E^s)
\quad\text{and}\quad
DG_k\bigl(\mathcal{C}_\varepsilon(E^{u},E^s)\bigr)\subset \mathcal{C}_\varepsilon(E^{u},E^s),
$$
Consequently, $G_k$ is partially hyperbolic for every $k \geq K(\varepsilon)$ with partially hyperbolic splitting
$$
T\mathbb{T}^3 = F^{uu}_k \oplus_{\succ} F^{cu}_k \oplus_{\succ} F^{cs}_k
$$
such that
$$
F^{cu}_k\subset \mathcal{C}_\varepsilon(E^{u},E^s),\quad  F^{cs}_k\subset\mathcal{C}_\varepsilon(E^s,E^u) \quad\text{and}\quad  F^{cu}_k\oplus F^{cs}_k=E^u\oplus E^s.
$$
\end{lem}

In what follows, we let $\varepsilon$ in Lemma~\ref{jixian2} so that
$$
2\varepsilon^2 \;\le\; \kappa_2^2.
$$
For   any $k\ge K(\varepsilon)$  in the Lemma~\ref{jixian2}, by combining the criterion in Subsection~\ref{criter} with the proof of Theorem~\ref{main2}—in a similar manner, except that we now consider the projection
$$
\pi^x_{uu} : \mathcal{F}^{uu}(x) \times \mathcal{F}^u(x) \times \mathcal{F}^{ss}(x) \to \mathcal{F}^{uu}(x), \quad \pi^x_{uu}(x,y,z) = x,
$$
—we obtain that for every ergodic Gibbs $u$-state $\mu$ of $G_k$, 
$$
\mu(\Lambda(q_i)) \leq\; \frac{(1+\kappa^2_2)^{3/2}}{100} + \frac{1}{100} \quad \text{for each } i=1,2.
$$
Similar to relation~\ref{contro}, when $x \notin \Lambda(q_1)$, we have 
$$
 \left| \det\left( DG_k \big|_{F^{cu}_k} \right) \right|\ge \frac{1}{\sqrt{1+\kappa_2^2}} \lambda_{u}.
$$
 and when $x \notin \Lambda(q_2)$, we have
{\color{black}\begin{equation}\label{cs}
 \left| \det\left( DG_k \big|_{F^{cs}_k} \right) \right|\le \sqrt{1+\kappa_2^2} {\color{black}\lambda_{ss}}
\end{equation}
(\textbf{Explanation:} Keep in mind that, whenever $x \notin \Lambda(q_2)$,
$$
DG_k(F^{cs}_k) = F^{cs}_k \in \mathcal{C}_\varepsilon(E^s, E^u), 
\quad \text{and} \quad DG_k(E^u) = E^u.)
$$ Then, by an argument similar to that for relation~\ref{contro}, we obtain relation~\ref{cs}.
}

We shall now prove the following statement.
\begin{lem}\label{main3lem}
There exists $K$ such that every $k\ge K$ ($\ge K(\varepsilon)$),  $F^{cu}_k$ is mostly expanding (but not uniformly expanding)  and $F^{cs}_k$ is mostly contracting (but not uniformly contracting).
\end{lem}
\begin{proof}
When $x\in\Lambda(q_1)$, for $v=(0,1,\epsilon_1)\in F^{cu}_k$, we can check that
\[
\frac{\|DG_k(v)\|}{\|v\|}=\frac{\sqrt{(\frac{\partial Q_1}{\partial b}+\frac{\partial Q_1}{\partial c}\cdot\epsilon_1)^2+(\lambda_{ss}\epsilon_1)^2}}{\sqrt{\epsilon_1^2+1}}.
\]
{\color{black} When $x\in\Lambda(q_2)$, for $v=(0,\epsilon_2, 1)\in F^{cs}_k$, we can check that
\[
\frac{\|DG_k(v)\|}{\|v\|}=\frac{\sqrt{((\frac{\partial Q_2}{\partial c})^{-1}+*\cdot\epsilon_2)^2+(\epsilon_2\lambda_u)^2}}{\sqrt{\epsilon_2^2+1}}.
\]}
Since both $\epsilon_1$ and $\epsilon_2$ are less than or equal to $\varepsilon$ in Lemma~\ref{jixian2},  {\color{black}  and $\frac{\partial Q_1}{\partial c}$ and $*$} both converge to zero as $k \to +\infty$, we can choose $\varepsilon$ sufficiently small such that
$$
\left| \det\left( DG_k \big|_{F^{cu}_k} \right) \right| \geq \frac{1}{2} - \kappa_2 \quad \text{when} \quad x \in \Lambda(q_1)
$$
{\color{black}$$
\left| \det\left( DG_k \big|_{F^{cs}_k} \right) \right| \leq 2 +\kappa_2 \quad \text{when} \quad x \in \Lambda(q_2).
$$}
It follows from relation~\ref{kappa2} that
$$
\int_{\Lambda(q_1)\cup (\T^3\setminus\Lambda(q_1))}\log(\det DG_k|_{F^{cu}_k})d\mu\ge (\frac{(1+\kappa^2_2)^{\frac{3}{2}}}{100}+\frac{1}{100})\log(\frac{1}{2}-\kappa_2)+(\frac{99}{100}-\frac{(1+\kappa_2^2)^{\frac{3}{2}}}{100})\log(\frac{1}{\sqrt{1+\kappa_2^2}} \lambda_{u} )>0.
$$
{\color{black} It follows from relation~\ref{kappa33} that
$$
\int_{\Lambda(q_2)\cup (\T^3\setminus\Lambda(q_2))}\log(\det DG_k|_{F^{cs}_k})d\mu\le  (\frac{(1+\kappa^2_2)^{\frac{3}{2}}}{100}+\frac{1}{100})\log(2+\kappa_2)+(\frac{99}{100}-\frac{(1+\kappa_2^2)^{\frac{3}{2}}}{100})\log(\sqrt{1+\kappa_2^2} {\color{black}\lambda_{ss}})<0.
$$}

We can directly verify that
$$
DG_k(q_1) = 
\begin{pmatrix}
\lambda_{uu} & 0 & 0\\[3pt]
0 & \frac{1}{2} & 0\\[3pt] 
0 & 0 & \lambda_{ss}
\end{pmatrix}
\quad \text{and} \quad 
DG_k(q_2) = 
\begin{pmatrix}
\lambda_{uu} & 0 & 0\\[3pt]
0 & \lambda_{u} & 0\\[3pt] 
0 & 0 & 2
\end{pmatrix},
$$
where $p_1, p_2$ remain fixed points. Hence, it follows directly that $F^{cu}_k$ is not uniformly expanding, and $F^{cs}_k$ is not uniformly contracting.
\end{proof}

\begin{proof}[Proof of Theorem~\ref{main3}]
By Lemma~\ref{main3lem}, we can directly find $G$. Since the mixed center is $C^{1+}$-robust, we can thus complete the proof.
\end{proof}

{\color{black}
We point out that our approach and that of Katok and Hasselblatt \cite[Part~4]{KatokHasselblatt1995} have subtle differences regarding the construction of DA, based on the following facts: In the neighborhood of certain fixed points, we work in the inverse direction. From the forward perspective, this corresponds to modifying the contraction at the fixed point to an expansion along $E^{ss}$. However, unlike in \cite[Part~4]{KatokHasselblatt1995}, in our setting, one cannot directly modify the contraction to an expansion in the forward direction of the map.  For instance, if one were to modify in the forward direction of the map, then on $U_{q_2}$ one could set
$$
J_k(a,b,c) = \Bigl(a,\, b,\, \tfrac{1}{\lambda_{ss}} \cdot R_2(a,b,c)\Bigr),
$$
where
$$
R_2(a,b,c) 
= \phi(kc) \cdot \phi\bigl(\sqrt{a^2 + b^2}\bigr) \cdot \Bigl(2 - \lambda_{ss}\Bigr)c \;+\; \lambda_{ss} c.
$$
Let $r=\sqrt{a^2+b^2}$. We compute
$$
   \frac{\partial R_2}{\partial c}
   =  \bigl(kc\,\phi'(kc)+\phi(kc)\bigr)\,\phi(r)\,(2-\lambda_{ss})+\lambda_{ss}.
$$
The condition
$$
\frac{\partial R_2}{\partial c} > 0,
$$
which guarantees that the map is a diffeomorphism, is equivalent to
\begin{equation}\label{maodun}
\bigl(kc \, \phi'(kc) + \phi(kc)\bigr)\, \phi(r)\;>\;\frac{1}{1 - \tfrac{2}{\lambda_{ss}}}.
\end{equation}
Since the effective domain of $\phi$ actually depends on the fixed $\delta$, and since, by relation~\ref{budeng}, we may have 
$$
\inf \bigl\{\,(kc\,\phi'(kc)+\phi(kc))\,\phi(r)\,\bigr\} \;=\; -m,
$$
(it is straightforward to check that the derivative of $y \mapsto y\,\phi(y)$ can take negative values).Thus, if inequality~\ref{maodun} holds, then it necessarily follows that 
\begin{equation}\label{maodun2}
 \frac{1}{1 - \tfrac{2}{\lambda_{ss}}}\;<\;-m.
\end{equation}
However, because we require $\lambda_{ss}$ to be sufficiently small (see relation~\ref{kappa33}) after $\phi$ is chosen, this inequality~\ref{maodun2} eventually fails once $\lambda_{ss}$ becomes too small.

The definition of $I_k$ in Section~\ref{aaa} is analogous. In particular, in order to guarantee that the definition is well defined, if one uses the same $C^\infty$ bump function $\psi$, it is necessary to work with the inverse.
}

\section{Appendix}

In fact, we can prove the following more general result:
\begin{lem}\label{tuilun}
For any $\gamma>0$, there exist constants $\varepsilon_0 > 0$ and $M > 0$ such that
$$
\frac{c^2 + cu}{(1+\epsilon^2 + c^2)(1+\epsilon^2)} \ge M, \quad \text{for all } |u| \le \varepsilon_0, \ |\epsilon| \le \varepsilon_0, \text{ and } |c|\ge\gamma.
$$
\end{lem}
Lemma~\ref{paowu} can be considered a corollary of Lemma~\ref{tuilun}.
\begin{proof}
Notice that
$$
\frac{c^2 + cu}{(1+\epsilon^2 + c^2)(1+\epsilon^2)}=\frac{1+\frac{u}{c}}{(\frac{1}{c^2}+\frac{\epsilon^2}{c^2}+1)(1+\epsilon^2)}\ge\frac{1-\frac{|u|}{\gamma}}{(\frac{1}{\gamma^2}+\frac{\epsilon^2}{\gamma^2}+1)(1+\epsilon^2)}
$$
If we want the inequality
$$
\frac{c^2 + cu}{(1+\epsilon^2 + c^2)(1+\epsilon^2)} \;\ge\; M > 0
$$
to hold, it suffices to require that
$$
\varepsilon_0 = \tfrac{\gamma}{10}, 
\quad |\epsilon| \le \varepsilon_0, 
\quad |u| \le \varepsilon_0, 
\quad M = \tfrac{9}{10}\cdot \frac{1}{\Bigl(\tfrac{1}{\gamma^2}+\tfrac{101}{100}\Bigr)\Bigl(1+\tfrac{\gamma^2}{100}\Bigr)}.
$$
This concludes the proof.
\end{proof}

\section{Acknowledgements} 

We are grateful for the courses offered by Professor Liao and Professor Shi at the Tianyuan Mathematics Southeast Center at Xiamen University in 2025, which inspired some of the ideas behind the proofs.

\smallskip

E-mail address: zhanghangyue@nju.edu.cn

\end{document}